\documentclass[11pt,a4paper,oneside]{amsart}
\title{Derivations of octonion matrix algebras}
\author{Harry Petyt}

\usepackage{amsmath} % Lots of maths functionality
\usepackage{amssymb} % Maths symbols
\usepackage{amsthm} % Maths environments: \begin{proof}, etc.
\usepackage[english]{babel} % Language and hyphenation.
\usepackage[font=small,justification=centering]{caption} % More flexibility for captioning figures
\usepackage{enumitem} % Change enumeration labelling
\usepackage[T1]{fontenc} % For font encoding, to allow accents, copy & paste, inequality signs, etc. to all work nicely.
\usepackage[a4paper]{geometry} % Changing page shape
\usepackage[utf8]{inputenc} % To be loaded after fontenc, also for encoding.
\usepackage{mathabx} % Contains \pnest symbol and more. Clashes with accents for \ring
\usepackage{mathtools} % Uses amsmath, fixes quirks and adds functionality.
\usepackage[dvipsnames]{xcolor} % Allow links to have colour. Needs to be before hyperref and tikz-cd
\usepackage[pdftex,  colorlinks=true]{hyperref} % Makes references and citations into links
    \hypersetup{urlcolor=RoyalBlue, linkcolor=black,  citecolor=black}
\usepackage{tikz-cd} % Commutative diagrams
\usepackage{xfrac} % Nicer quotients 
\usepackage{multirow}

\makeatletter\def\subsection{\@startsection{subsection}{1}\z@{.7\linespacing\@plus\linespacing}
    {.5\linespacing}{\normalfont\scshape\centering}}\makeatother % centrally aligned subsections

\newcounter{claimcount}
\newcounter{thmcount}

\newtheorem{theorem}{Theorem}[section]
\newtheorem{lemma}[theorem]{Lemma}

\newtheorem{proposition}[theorem]{Proposition}
\newtheorem*{theorem*}{Theorem}
\theoremstyle{definition}

\newtheorem{question}[theorem]{Question}

\newenvironment{claim}[1]{\stepcounter{claimcount}\par\noindent\underline{Claim \theclaimcount:}\space#1}{}
\newenvironment{claimproof}[1]{\par\noindent\underline{Proof:}\space#1}
    {\leavevmode\unskip\penalty9999\hbox{}\nobreak\hfill\quad\hbox{$\diamondsuit$}\vspace{2mm}}

\newcommand{\showcomments}{yes}
\newsavebox{\commentbox}

\renewcommand{\bar}{\overline}

\newcommand*{\mfa}{\mathfrak{a}}
\DeclareMathOperator{\adjoint}{ad}
\newcommand*{\C}{\mathbb{C}}
\DeclareMathOperator{\der}{der}
\newcommand*{\F}{\mathbb{F}}
\newcommand*{\mfg}{\mathfrak{g}_2}
\newcommand*{\HH}{\mathbb{H}}
\newcommand*{\mfh}{\mathfrak{h}}
\DeclareMathOperator{\imag}{Im}

\newcommand*{\OO}{\mathbb{O}}
\DeclareMathOperator{\pure}{Pu}

\newcommand*{\R}{\mathbb{R}}
\DeclareMathOperator{\real}{Re}

\newcommand{\vcc}{\hspace{1mm}\vcentcolon\hspace{1mm}}

\newcommand{\ignore}[2]{\left\{\kern-.7ex\left\{#1\right\}\kern-.7ex\right\}_{#2}}

\tikzset{symbol/.style={draw=none,every to/.append style={edge node={node [sloped, allow upside down, auto=false]{$#1$}}}}}

\begin{document}

\maketitle

\begin{abstract}
It is well-known that the exceptional Lie algebras $\mathfrak{f}_4$ and $\mathfrak{g}_2$ arise from the octonions as the derivation algebras of the $3\times3$ hermitian and $1\times1$ antihermitian matrices, respectively. Inspired by this, we compute the derivation algebras of the spaces of hermitian and antihermitian matrices over an octonion algebra in all dimensions. 
\end{abstract}

\section{Introduction}

In \cite{benkartosborn:derivations}, Benkart and Osborn calculate the derivation algebra for the algebra of $n\times n$ matrices with entries in an arbitrary unital algebra, under the standard matrix product, the commutator product, and the anticommutator product. In the case that the unital algebra is an octonion algebra over a field $\F$, their results show that for both the standard product and the anticommutator, the derivation algebra is $\mfg\oplus\mathfrak{gl}_n(\F)$; while in the commutator case it is the direct sum of this with $\F$.

The exceptional Lie algebra $\mathfrak{f}_4$ can be constructed as the derivation algebra of the exceptional Jordan algebra, which is the set of $3\times3$ hermitian matrices with entries in an octonion algebra, under the anticommutator product. If we increase the size of these matrices then we lose the Jordan property but still get well defined algebras. It is then natural to ask what the corresponding derivation algebras are, and to do the same for antihermitian (or skew-hermitian) matrices. When the characteristic of $\F$ is not two, our answers are:
\begin{theorem*}
If $n\geq4$ then $\der(\mfh_n(\OO))=\mfg\oplus\mathfrak{so}_n(\F)$. 
\end{theorem*}
\begin{theorem*}
$\der(\mfa_n(\OO))=\mfg\oplus\mathfrak{so}_n(\F)$ for all natural numbers $n$.
\end{theorem*}

This is strongly reminiscent of Benkart and Osborn's results, so since every matrix decomposes as the sum of a hermitian matrix with an antihermitian matrix, one might hope to use their methods. In practice, however, many of the tools they use break down in our case. This is mostly because, for us, entries on the diagonal come from a subspace of the octonion algebra.

I would like to thank Dmitriy Rumynin for many useful discussions, and without whom the present work would not have been possible. I would also like to thank Ivan Shestakov for valuable information.

\section{Set-up}

Let $\mathbb{F}$ be a field of characteristic not two, and let $\OO$ be an octonion algebra over $\F$. That is, $\OO$ is a unital, alternative, 8-dimensional $\F$-algebra with a nondegenerate quadratic form $|\cdot|^2\vcentcolon\OO\rightarrow\F$ which is multiplicative in the sense that $|zw|^2=|z|^2|w|^2$. We call elements of $\OO$ \emph{octonions}. The nucleus of $\OO$ is $\F$. More than being alternative, the nonzero octonions form a Moufang loop under multiplication. In particular they satisfy the left Moufang law
\begin{align}
z(w(zu))=(zwz)u. \label{moufanglaw}
\end{align}

We distinguish two cases for $\OO$. We say that $\OO$ is \emph{Type~I} if it has an orthonormal basis $1,e_1,\dots,e_7$ such that $e_i^2=-1$ for all $i$, and \emph{Type~II} otherwise. From work of Jacobson \cite[Sect. 3]{jacobson:composition}, if $\OO$ is Type~II then it is split and has an orthonormal basis $1,e_1,\dots,e_7$ such that \hspace{1mm} i) $e_i^2=-1$ for $i\leq3$ and $e_i^2=1$ otherwise; and \hspace{1mm} ii) the $\F$-span of $1,e_1,e_2,e_3$ is isomorphic to $_\F\HH$, the quaternions over $\F$. These basis elements anticommute, and consequently 
\begin{align}
e_ie_je_i=-e_ie_ie_j=\pm e_j. \label{bimultiplication}
\end{align} 
Note that being Type~I does not mean that $\OO$ is a division algebra -- consider $\F=\C$, for example (for a classification of octonion algebras, see \cite{serre:cohomologie}). In any case, the algebra $\der(\OO)$ is simple of type $\mfg$ \cite[Thm. 6]{jacobson:composition}.

We denote conjugation in $\OO$ by a bar: if $z=z_0+\sum_{i=1}^7z_ie_i$ then $\bar{z}=z_0-\sum_{i=1}^7z_ie_i$. We write $\real(z)=z_0$ and call it the \emph{real part} of $z$, even when the base field is not $\R$. Likewise, we call $\imag(z)=\sum_{i=1}^7z_ie_i$ the imaginary part of $z$, and we denote the set of all such purely imaginary octonions by $\pure(\OO)$. For more on octonions, see \cite{baez:octonions,conwaysmith:onquaternions,springerveldkamp:octonions}.

We are interested in certain spaces of matrices with entries in $\OO$. Such a matrix $x$ has a conjugate, $x^*$, which is obtained from $x$ by taking the transpose and conjugating all the entries. If $x^*=x$ then we call it \emph{hermitian}, and we denote the set of hermitian $n\times n$ matrices with entries in $\OO$ by $\mfh_n(\OO)$. The anticommutator product $x\circ y=xy+yx$ makes $\mfh_n(\OO)$ into an $\F$-algebra. Similarly, if $x^*=-x$ then we say $x$ is \emph{antihermitian}, and we write $\mfa_n(\OO)$ for the set of antihermitian matrices, which is made into an $\F$-algebra by the commutator product $[x,y]=xy-yx$. We write $E_{ij}$ for the matrix with a 1 in the $ij^\mathrm{th}$ place and zeros everywhere else.

Our aim is to calculate the derivation algebras $\der(\mfh_n(\OO))$ and $\der(\mfa_n(\OO))$. Note that $\mfh_1(\OO)=\F$ and $\mfa_1(\OO)=\pure(\OO)$, so $\der(\mfh_1(\OO))=0$ and $\der(\mfa_1(\OO))=\mfg$. 

Let $\mathfrak{so}_n(\F)$ be the algebra of antisymmetric $n\times n$ matrices with entries in $\F$ under the commutator product. We use this to state an important lemma.

\begin{lemma} \label{g2son}
If $n>1$ then both $\der(\mfh_n(\OO))$ and $\der(\mfa_n(\OO))$ have a subalgebra isomorphic to $\mfg\oplus\mathfrak{so}_n(\F)$.
\end{lemma}

\begin{proof}
An element of $\mfg$ gives a derivation by acting on a matrix $x$ entrywise. The action of $\mathfrak{so}_n(\F)$ is the adjoint action: if $A\in\mathfrak{so}_n(\F)$ then $\adjoint_A\vcentcolon x\mapsto[A,x]$ is a derivation because $A$ has entries in $\F$, meaning that $A^*=A^T=-A$. It is easy to check that the two actions commute, whence the direct summation.
\end{proof}

\section{Hermitian}

For hermitian matrices, in dimensions 2 and 3 Jacobson tells us \cite[Thm. 14]{jacobson:some}:

\begin{theorem} If the characteristic of $\F$ is not two or three then $\der(\mfh_2(\OO))=\mathfrak{so}_9(\F)$ and $\der(\mfh_3(\OO))=\mathfrak{f}_4$.
\end{theorem}

This result extends earlier work of Chevalley and Schafer over algebraically closed fields of characteristic zero \cite{chevalleyschafer:exceptional}, and is the main motivation for the present work.

We henceforth consider $n\geq4$ only in this section, and allow the characteristic of $\F$ to be three. There are five types of nonzero product in $\mfh_n(\OO)$:
\begin{gather}
E_{ii}\circ E_{ii}=2E_{ii} \label{hiiii} \\
E_{ii}\circ (zE_{ij}+\bar{z}E_{ji}) =zE_{ij}+\bar{z}E_{ji} \label{hiiij} \\
E_{jj}\circ (zE_{ij}+\bar{z}E_{ji}) =zE_{ij}+\bar{z}E_{ji} \label{hjjij} \\
(zE_{ij}+\bar{z}E_{ji})\circ(wE_{ij}+\bar{w}E_{ji}) 
	=2\real(z\bar{w})(E_{ii}+E_{jj}) \label{hijij} \\
(zE_{ij}+\bar{z}E_{ji})\circ(wE_{jk}+\bar{w}E_{kj}) 
	=zwE_{ik}+\bar{w}\bar{z}E_{ki}. \label{hijjk}
\end{gather}

By applying a derivation $\partial$ to these we can obtain constraints that $\partial$ must satisfy. We first do this for a special subset of derivations.

\begin{proposition} \label{htraceless}
If $n\geq4$ then the subalgebra of derivations $\partial\vcentcolon\mfh_n(\OO)\rightarrow\mfh_n(\OO)$ such that $\partial(E_{ii})=0$ for all~$i$ is isomorphic to $\mathfrak{g}_2$.
\end{proposition}

\begin{proof}
From \eqref{hiiij} and \eqref{hjjij} we have 
\[
E_{ii}\circ\partial(zE_{ij}+\bar{z}E_{ji}) =\partial(zE_{ij}+\bar{z}E_{ji}) =E_{jj}\circ\partial(zE_{ij}+\bar{z}E_{ji}),
\]
and hence there are linear maps $\alpha^{ij}\vcentcolon\OO\rightarrow\OO$ such that
\begin{align*}
\partial(zE_{ij}+\bar{z}E_{ji})=\alpha^{ij}(z)E_{ij}+\overline{\alpha^{ij}(z)}E_{ji}.
\end{align*}
Since $\partial(E_{ii})=0$ for all $i$, the $\alpha^{ij}$ determine $\partial$. Note in particular that 
\begin{align}
\alpha^{ij}(z)=\overline{\alpha^{ji}(\bar{z})}. \label{alphaijalphaji}
\end{align} 
Applying $\partial$ to \eqref{hijij}, we find in the $ii^\mathrm{th}$ place the equality 
\[
\alpha^{ij}(z)\bar{w} +w\overline{\alpha^{ij}(z)} +z\overline{\alpha^{ij}(w)} +\alpha^{ij}(w)\bar{z}=0,
\] 
which we can restate as
\begin{align}
\real(\alpha^{ij}(z)\bar{w})+\real(\alpha^{ij}(w)\bar{z})=0. \label{realalphazw}
\end{align}
Letting $z=w$ run through the standard basis of $\OO$ in \eqref{realalphazw} we get that both the real part of $\alpha^{ij}(1)$ and the $e_k^\mathrm{th}$ part of $\alpha^{ij}(e_k)$ are zero. By taking $z=e_k\neq e_l=w$ in \eqref{realalphazw} we get that the $e_k^\mathrm{th}$ part of $\alpha^{ij}(e_l)$ is the negative of the $e_l^\mathrm{th}$ part of $\alpha^{ij}(e_k)$. With $z=1$, $w=e_k$ in \eqref{realalphazw} we get two cases. If $\OO$ is Type~I then  the $e_k^\mathrm{th}$ part of $\alpha^{ij}(1)$ is the negative of the real part of either $\alpha^{ij}(e_k)$, and hence $\alpha^{ij}\in\mathfrak{so}_8(\F)$. On the other hand, if $\OO$ is Type~II then the $e_k^\mathrm{th}$ part of $\alpha^{ij}(1)$ is equal to the real part of either $\alpha^{ij}(e_k)$ or its negative, depending on whether $k>3$ or $k\leq3$, respectively. Thus the matrix of $\alpha^{ij}$ has the form 
\begin{align}
\left( \begin{array}{c|cc}
	0 & -v_1^T & v_2^T \\ \hline v_1 & & \\
	v_2 & \multicolumn{2}{c}{\smash{\raisebox{.4\normalbaselineskip}{$A$}}}   \end{array} \right) \in \F^{8\times8}, \label{TypeIImatrix}
\end{align} 
where $v_1\in\F^3$, $v_2\in\F^4$, and $A\in\mathfrak{so}_7(\F)$. We now apply $\partial$ to \eqref{hijjk}, and find in the $ik^\mathrm{th}$ place the equality 
\begin{align}
\alpha^{ik}(zw)=\alpha^{ij}(z)w+z\alpha^{jk}(w). \label{hprederivation}
\end{align}
In particular, if $e_t\neq e_r$ then  
\[
\alpha^{ik}(e_te_r)=\alpha^{ij}(e_t)e_r+e_t\alpha^{jk}(e_r) \hspace{5mm} \mathrm{and} \hspace{5mm} \alpha^{ij}(e_te_r)=\alpha^{ik}(e_t)e_r+e_t\alpha^{kj}(e_r).
\]
Comparing $e_t^\mathrm{th}$ parts, it follows from the form \eqref{TypeIImatrix} of $\alpha^{ij}$ that if $\OO$ is Type~II and $r>3$ then $\real(\alpha^{jk}(e_r))=\real(\alpha^{kj}(e_r))$, and so is zero by \eqref{alphaijalphaji}. In particular, irrespective of whether $\OO$ is Type~I or Type~II, we have 
\begin{align}
\alpha^{ij}\in\mathfrak{so}_8(\F). \label{soeight}
\end{align}

Returning to \eqref{hprederivation}, the maps $\alpha^{ik}, \alpha^{ij}, \alpha^{jk}$ are said to be in triality, and in light of \eqref{soeight}, any one uniquely determines the other two \cite[p.42]{springerveldkamp:octonions}. We use these trialities to show that all $\alpha^{ij}$ are equal. 

If $j>2$ then we have trialities $\alpha^{12},\alpha^{1j},\alpha^{j2}$. Since these share the same first map we have $\alpha^{1j}=\alpha^{13}$ and $\alpha^{2j}=\alpha^{23}$ whenever $j>2$.

If $k>j>2$ then the trialities $\alpha^{1j},\alpha^{1k},\alpha^{kj}$ share the same first map, so all $\alpha^{jk}$ with $k>j>2$ are equal to $\alpha^{34}$.

The two trialities $\alpha^{12},\alpha^{14},\alpha^{42}$ and $\alpha^{13},\alpha^{14},\alpha^{43}$ share the same second map, so $\alpha^{12}=\alpha^{13}$ and $\alpha^{24}=\alpha^{34}$. Hence if $k>j>1$ then all $\alpha^{1j}$ are equal to $\alpha^{12}$, and all $\alpha^{jk}$ are equal to~$\alpha^{23}$.

Finally, the two trialities $\alpha^{13},\alpha^{14},\alpha^{43}$ and $\alpha^{23},\alpha^{24},\alpha^{43}$ share the same third map, so $\alpha^{14}=\alpha^{24}$. Hence all the $\alpha^{ij}$ are equal to $\alpha^{12}$. Writing $\alpha=\alpha^{12}$, we can now read \eqref{hprederivation} as
\begin{align*}
\alpha(zw)=\alpha(z)w+z\alpha(w).
\end{align*}
That is, $\alpha\in\der\OO=\mfg$, and $\partial$ is given by applying $\alpha$ to each entry.
\end{proof}

\begin{theorem} \label{hermitianderivations}
If $n\geq4$ and the characteristic of $\F$ is not two, then $\der(\mfh_n(\OO))=\mfg\oplus\mathfrak{so}_n(\F)$.
\end{theorem}

\begin{proof}
Let $\partial$ be a derivation. Our strategy is to show that $\partial$ differs from one of the derivations of Proposition \ref{htraceless} by the adjoint action of an element of $\mathfrak{so}_n(\F)$. Applying $\partial$ to \eqref{hiiii} we find that there are constants $\mu_{ik}^i\in\OO$ for $k\neq i$ such that
\begin{align*}
\partial(E_{ii})=\sum_{k\neq i} (\mu_{ik}^iE_{ik}+\overline{\mu_{ik}^i}E_{ki}),
\end{align*}
and $\partial$ applied to $E_{ii}\circ E_{jj}=0$ yields
\begin{align}
\mu_{ij}^i=-\overline{\mu_{ji}^j}. \label{muantisymmetry}
\end{align}
The $\partial(E_{ii})$ are thus determined by the choice of $\mu_{ij}^i$ with $j>i$. Now let 
\[
\partial(zE_{ij}+\bar{z}E_{ji})=\sum_{k,l}\alpha_{kl}^{ij}(z)E_{kl}.
\] 
Similarly to in the proof of Proposition \ref{htraceless}, applying $\partial$ to \eqref{hiiij} and \eqref{hjjij} leads to
\begin{align*}
\partial(zE_{ij}+\bar{z}E_{ji})
	=2E_{jj}&\real(z\mu_{ij}^i)-2E_{ii}\real(\mu_{ij}^i\bar{z}) 
		+\alpha_{ij}^{ij}(z)E_{ij}+\overline{\alpha_{ij}^{ij}(z)}E_{ji}	\\
		+&\sum_{t\neq i,j}\big( z\mu_{jt}^jE_{it}
		+\overline{\mu_{jt}^j}\bar{z}E_{ti}+\overline{\mu_{it}^i}zE_{tj}
		+\bar{z}\mu_{it}^iE_{jt}\big). \nonumber
\end{align*}
In particular, if the pair $(k,l)$ is not equal to either $(i,j)$ or $(j,i)$ then $\alpha_{kl}^{ij}=0$. We therefore abbreviate $\alpha_{ij}^{ij}$ to just $\alpha^{ij}$, and note that $\partial$ is determined by the $\mu_{ij}^i$ and the $\alpha^{ij}$ with $j>i$. Now apply $\partial$ to both sides of \eqref{hijjk}. If $t\neq i,j,k$ then in the $it^\mathrm{th}$ place we find the equality
\begin{align*}
(zw)\mu_{kt}^k=z(w\mu_{kt}^k).
\end{align*}
Thus $\mu_{kt}^k$ lies in the nucleus of $\OO$, which is $\F$. By varying $i,j,k$ we find that all $\mu_{ij}^i$ lie in~$\F$.

Let $A=(\mu_{ij}^i)_{ij}$. Since all $\mu_{ij}^i$ lie in $\F$, equation \eqref{muantisymmetry} tells us that $A\in\mathfrak{so}_n(\F)$, so $\adjoint_A\in\der(\mfh_n(\OO))$. Moreover, 
\[
\adjoint_A(E_{ii})=\sum_{k\neq i}(\mu_{ik}^iE_{ik}+\overline{\mu_{ik}^i}E_{ki})=\partial(E_{ii}).
\]
Hence $\partial-\adjoint_A$ is a derivation which maps all $E_{ii}$ to zero, so by Proposition \ref{htraceless}, $\partial-\adjoint_A$ is given by an element of $\mfg$, and by Lemma \ref{g2son} we are done.
\end{proof}

The exceptional Lie algebra $\mathfrak{e}_6$ can be constructed as 
\[
\mathfrak{e}_6=\der(\mfh_3(\OO))+\{L_x\vcc x\in\mfh_3(\OO), \hspace{1mm} \mathrm{tr}(x)=0\},
\]
where $L_x$ denotes left multiplication by $x$. This is due to Chevalley and Schafer \cite{chevalleyschafer:exceptional} (see also \cite[Sect. 4.4]{schafer:introduction}). A natural question therefore arises from Theorem~\ref{hermitianderivations}:
\begin{question}
How does this construction of $\mathfrak{e}_6$ generalise to $\mfh_n(\OO)$?
\end{question}
One barrier to generalisation is that the commutator of two left multiplications may fail to be a derivation, for while in the $3\times3$ case the derivation algebra has dimension $\dim(\mathfrak{f}_4)=52$, in the $4\times4$ case its dimension is only $\dim(\mfg\oplus\mathfrak{so}_4(\F))=20$. One remedy would be to include products of multiplication maps, and some work in this vein is done in \cite{petyt:special}.

\section{Antihermitian}

Here we compute the algebras $\der(\mfa_n(\OO))$ for all $n$. We find this to be more fiddly than the hermitian case. Again there are five types of nonzero product in $\mfa_n(\OO)$:
\begin{gather}
[e_iE_{tt},e_jE_{tt}]=2e_ie_jE_{tt} \hspace{5mm} i\neq j \label{atttt} \\
[e_iE_{tt},zE_{tr}-\bar{z}E_{rt}]=e_izE_{tr}+\bar{z}e_iE_{rt} \label{atttr} \\
[e_iE_{rr},zE_{tr}-\bar{z}E_{rt}]=-ze_iE_{tr}-e_i\bar{z}E_{rt} \label{arrtr} \\
[zE_{tr}-\bar{z}E_{rt},wE_{tr}-\bar{w}E_{rt}]= 2\imag(w\bar{z})E_{tt}+2\imag(\bar{w}z)E_{rr} \label{atrtr} \\
[zE_{tr}-\bar{z}E_{rt},wE_{rs}-\bar{w}E_{sr}]=zwE_{ts}-\bar{w}\bar{z}E_{st}. \nonumber
\end{gather}
and again we find restrictions on a derivation by applying it to (the first four of) these.

\begin{theorem}
If the characteristic of $\F$ is not two, then $\der(\mfa_n(\OO))=\mfg\oplus\mathfrak{so}_n(\F)$ for all natural numbers $n$.
\end{theorem}

\begin{proof}
By Lemma \ref{g2son} it suffices to bound the dimension of $\der(\mfa_n(\OO))$ above by $14+\frac{n(n-1)}{2}$. Applying a derivation $\partial$ to both sides of \eqref{atttt}, we find that for $k\neq t$ there are linear maps $a_{tk}^t\vcentcolon\pure(\OO)\rightarrow\OO$ and $a_{tt}^t\vcentcolon\pure(\OO)\rightarrow\pure(\OO)$ such that 
\begin{align*}
\partial(e_iE_{tt})=a_{tt}^t(e_i)E_{tt}+ \sum_{k\neq t} \big(a_{tk}^t(e_i)E_{tk}-\overline{a_{tk}^t(e_i)}E_{kt}\big),
\end{align*}
and moreover, if $k\neq t$ and $i\neq j$ then 
\begin{align}
2a_{tk}^t(e_ie_j)=e_ia_{tk}^t(e_j)-e_ja_{tk}^t(e_i). \label{aij}
\end{align}
Applying $\partial$ to $[e_iE_{tt},e_iE_{rr}]=0$ we get that $a_{tr}^t(e_i)e_i+e_i\overline{a_{rt}^r(e_i)}=0$, and hence
\begin{align}
a_{rt}^r(e_i)=\pm e_i\overline{a_{tr}^t(e_i)}e_i. \label{artrtrt}
\end{align} 
Now let $\partial(zE_{tr}-\bar{z}E_{rt})=\sum\beta_{kl}^{tr}(z)E_{kl}$. Applying $\partial$ to \eqref{atttr}, we get in positions $rr$, $tr$, $tk$, $rk$, $kl$ (with $k,l,r,t$ pairwise distinct) the following respective equalities:
\begin{gather}
\beta_{rr}^{tr}(e_iz)=2\imag(\bar{z}a_{tr}^t(e_i)) \label{tbrr} \\
\beta_{tr}^{tr}(e_iz)=e_i\beta_{tr}^{tr}(z)+a_{tt}^t(e_i)z \label{tbtr} \\
\beta_{tk}^{tr}(e_iz)=e_i\beta_{tk}^{tr}(z) \label{tbtk} \\
\beta_{rk}^{tr}(e_iz)=\bar{z}a_{tk}^t(e_i) \label{tbrk} \\
\beta_{kl}^{tr}(e_iz)=0. \label{tbkl}
\end{gather}
\begin{claim}{} \label{alambda}
The $a_{tr}^t$ are scalar multiples of the identity map $I\vcentcolon \pure(\OO)\rightarrow\pure(\OO)$.
\end{claim}
\begin{claimproof}
Taking $z=e_i$ in \eqref{tbrr} gives $\beta_{rr}^{tr}(1)=\pm2\imag(e_ia_{tr}^t(e_i))$, depending on whether $\OO$ is Type~I or Type~II and on the value of $i$, and it follows that all but the $e_i^\mathrm{th}$ part of $a_{tr}^t(e_i)$ is determined by $a_{tr}^t(e_1)$. Comparing $(e_ie_j)^\mathrm{th}$ parts in \eqref{aij} we find that the $(e_ie_j)^\mathrm{th}$ part of $a_{tk}^t(e_ie_j)$ is half the sum of the $e_j^\mathrm{th}$ part of $a_{tk}^t(e_j)$ with the $e_i^\mathrm{th}$ part of $a_{tk}^t(e_i)$. Cycling $e_i\mapsto e_j\mapsto e_ie_j\mapsto e_i$, we get that the  $e_i^\mathrm{th}$ part of $a_{tk}^t(e_i)$ is the same for all~$i$.

Thus, if we set $c_{tr}^t=a_{tr}^t+\real(e_1a_{tr}^t(e_1))I$ then $c_{tr}^t$ is simply $a_{tr}^t$ except that the $e_i^\mathrm{th}$ part of $c_{tr}^t(e_i)$ is zero for all $i$, and by \eqref{tbrr} we have $2e_ic_{tr}^t(e_i)=\beta_{rr}^{tr}(1)$. In particular, 
\begin{align}
c_{tr}^t(e_j)=\left\{\begin{array}{@{}l@{}l}
\multirow{2}{*}{$-e_j(e_ic_{tr}^t(e_i))$} &\hspace{5mm} \text{if }\OO \text{ is Type I} \\
 &\hspace{7mm} \text{or }\OO \text{ is Type II and }j\leq3; \\[1mm]
e_j(e_ic_{tr}^t(e_i)) &\hspace{5mm} \text{if }\OO \text{ is Type II and }j>3.
\end{array}\right. \label{citoj}
\end{align}
Since $c_{tr}^t$ differs from $a_{tr}^t$ only by a multiple of the identity, \eqref{aij} holds for $c_{tr}^t$. If $\OO$ is Type~II and either $i\leq3<j$ or $j\leq3<i$, then combining \eqref{aij} with \eqref{citoj} gives
\[
2c_{tr}^t(e_ie_j)=e_ic_{tr}^t(e_j)-e_jc_{tr}^t(e_i)= \pm e_i(e_j(e_ic_{tr}^t(e_i)))-e_jc_{tr}^t(e_i)=0,
\]
where the last equality holds by the left Moufang law \eqref{moufanglaw} and equation \eqref{bimultiplication}. On the other hand, if $i,j\leq3$, $i,j>3$, or $\OO$ is Type~I then
\begin{align}
2c_{tr}^t(e_ie_j)=e_ic_{tr}^t(e_j)-e_jc_{tr}^t(e_i)
=-2e_i(e_j(e_ic_{tr}^t(e_i))) \label{typeiiandijparity}
\end{align}
But \eqref{citoj} tells us that $2c_{tr}^t(e_ie_j)=-2(e_ie_j)(e_ic_{tr}^t(e_i))$, so the associator $[e_i,e_j,e_ic_{tr}^t(e_i)]=0$, and hence $c_{tr}^t(e_i)$ lies in the subalgebra generated by $e_i$ and $e_j$. If $\OO$ is Type~I then this holds for all choices of $e_j$, so $c_{tr}^t(e_i)$ lies in the complex subalgebra generated by $e_i$. But we constructed $c_{tr}^t$ so that the $e_i^\mathrm{th}$ part of $c_{tr}^t(e_i)$ is zero, and hence $c_{tr}^t(e_i)\in\F$. Now comparing real parts in \eqref{aij}, we find that the real part of $c_{tr}^t(e_ie_j)$ is zero. Hence $c_{tr}^t=0$ if $\OO$ is Type~I. 

Similarly, if $\OO$ is Type~ II and $i,j,k>3$ then $c_{tr}^t(e_i)$ lies in both the subalgebra generated by $e_i$ and $e_j$ and in the subalgebra generated by $e_i$ and $e_k$, so it lies in the subalgebra generated by $e_i$, and hence is an element of $\F$. If $i\leq 3$ then we can partition $\{e_4,e_5,e_6,e_7\}$ into two pairs $e_{j_1}, e_{k_1}$ and $e_{k_2},e_{k_2}$ such that $e_i=e_{j_l}e_{k_l}$. Then by \eqref{typeiiandijparity}, $c_{tr}^t(e_i)=-e_{j_l}(e_{k_l}(e_{j_l}c_{tr}^t(e_{j_l})))$, which since $j_l>3$ lies in the $\F$-span of $e_{k_l}$ for both $l=1$ and $l=2$. Hence $c_{tr}^t=0$ if $\OO$ is Type~II as well. It follows from the construction of $c_{tr}^t$ that there exist constants $\lambda_{tr}^t\in\F$ such that
\begin{align}
a_{tr}^t=\lambda_{tr}^tI\vcc \pure(\OO)\longrightarrow\pure(\OO),
\end{align}
which proves the claim.
\end{claimproof}

Applying $\partial$ to \eqref{arrtr}, we get in positions $tt$, $tr$, $tk$, $rk$ (with $k\neq r,t$) the following respective equalities: 
\begin{gather}
\beta_{tt}^{tr}(ze_i)=-2\imag(\lambda_{tr}^te_i\bar{z}) \label{rbtt} \\
\beta_{tr}^{tr}(ze_i)=za_{rr}^r(e_i)+\beta_{tr}^{tr}(z)e_i \label{rbtr} \\
\beta_{tk}^{tr}(ze_i)=za_{rk}^r(e_i) \label{rbtk} \\
\beta_{rk}^{tr}(ze_i)=-e_i\beta_{rk}^{tr}(z). \label{rbrk} 
\end{gather}
Taking $z=1$ in \eqref{tbrk} and \eqref{rbtk} and using Claim 1 gives 
\begin{gather*}
\beta_{rk}^{tr}\vert_{\pure(\OO)}=a_{tk}^t=\lambda_{tk}^tI, \hspace{5mm}
\beta_{tk}^{tr}\vert_{\pure(\OO)}=a_{rk}^r=\lambda_{rk}^rI,
\end{gather*}
and then by taking $z=e_i$ in \eqref{tbtk} and \eqref{rbrk} we conclude that 
\begin{gather}
\beta_{tk}^{tr}=\lambda_{rk}^rI, \hspace{5mm}
\beta_{rk}^{tr}=\lambda_{tk}^tI. \label{brklambda}
\end{gather}
Now, taking $z=1$ in \eqref{tbrr}, \eqref{tbtr}, \eqref{rbtt}, \eqref{rbtr} and $z=e_i$ in \eqref{tbrr} and \eqref{rbtt} gives:
\begin{gather}
\beta_{rr}^{tr}(e_i)=2\lambda_{tr}^te_i, \hspace{4mm} \beta_{rr}^{tr}(1)=0 \label{brrlambda} \\
a_{tt}^t(e_i)=\beta_{tr}^{tr}(e_i)-e_i\beta_{tr}^{tr}(1) \label{attt} \\
\beta_{tt}^{tr}(e_i)=-2\lambda_{tr}^te_i, \hspace{4mm} \beta_{tt}^{tr}(1)=0 \label{bttlambda} \\
a_{rr}^r(e_i)=\beta_{tr}^{tr}(e_i)-\beta_{tr}^{tr}(1)e_i. \label{arrr}
\end{gather}
\begin{claim}{}
There is an element $\beta\in\mfg$ such that $a_{tt}^t=\beta_{tr}^{tr}=\beta$ for all $t$ and $r$.
\end{claim}
\begin{claimproof}
Applying $\partial$ to \eqref{atrtr}, in the $tt^\mathrm{th}$ place we find the equality
\[
2a_{tt}^t(e_i)=\beta_{tr}^{tr}(1)e_i+ e_i\overline{\beta_{tr}^{tr}(1)}+ \beta_{tr}^{tr}(e_i)-\overline{\beta_{tr}^{tr}(e_i)}.
\]
Combining this with \eqref{attt} gives 
\[
2\beta_{tr}^{tr}(e_i) -2e_i\beta_{tr}^{tr}(1) =\beta_{tr}^{tr}(1)e_i +e_i\overline{\beta_{tr}^{tr}(1)} +2\imag(\beta_{tr}^{tr}(e_i)),
\]
and comparing $e_i^\mathrm{th}$ parts we find that
\begin{align}
\real(\beta_{tr}^{tr}(1))=0. \label{realbtr}
\end{align}
Taking $z=e_i$ in \eqref{rbtr} and using \eqref{arrr} leads to 
\[
\pm\beta_{tr}^{tr}(1)=\beta_{tr}^{tr}(e_i)e_i +e_i\beta_{tr}^{tr}(e_i) -e_i\beta_{tr}^{tr}(1)e_i,
\]
and by \eqref{realbtr}, if we compare real parts in this then we get that the $e_i^\mathrm{th}$~part of $\beta_{tr}^{tr}(e_i)$ is zero. Combining \eqref{tbtr} with \eqref{attt} and \eqref{rbtr} with \eqref{arrr} we get, respectively:
\begin{gather}
\beta_{tr}^{tr}(e_ie_j) =e_i\beta_{tr}^{tr}(e_j)+\beta_{tr}^{tr}(e_i)e_j-(e_i\beta_{tr}^{tr}(1))e_j \label{beiej} \\
\beta_{tr}^{tr}(e_ie_j) =e_i\beta_{tr}^{tr}(e_j)-e_i(\beta_{tr}^{tr}(1)e_j)+\beta_{tr}^{tr}(e_i)e_j, \nonumber
\end{gather}
and hence $(e_i\beta_{tr}^{tr}(1))e_j=e_i(\beta_{tr}^{tr}(1)e_j)$ for all $i$ and $j$, so $\beta_{tr}^{tr}(1)$ is in the nucleus of $\OO$, which is $\F$. By \eqref{realbtr} we now have $\beta_{tr}^{tr}(1)=0$, and it follows from \eqref{beiej} that
\begin{align*}
\beta_{tr}^{tr}(e_ie_j)=e_i\beta_{tr}^{tr}(e_j)+\beta_{tr}^{tr}(e_i)e_j.
\end{align*}
That is $\beta_{tr}^{tr}\in\der(\OO)=\mfg$. Using \eqref{attt} and \eqref{arrr} then varying $r$ and $t$ completes the proof of the claim.
\end{claimproof}

Putting together \eqref{tbkl}, \eqref{brklambda}, \eqref{brrlambda}, \eqref{bttlambda} with Claims 1 and 2 we find that $\partial$ is completely determined by the choice of $\beta\in\mfg$ and the $\lambda_{tk}^t$ with $k\neq t$. By \eqref{artrtrt} we only need $k>t$, so $\dim(\der(\mfa_n(\OO)))\leq14+\frac{n(n-1)}{2}$, and by Lemma \ref{g2son} we are done.
\end{proof}

\bibliographystyle{alpha}
\bibliography{bibtex}

\begin{thebibliography}{Bae02}

\bibitem[Bae02]{baez:octonions}
John~C. Baez.
\newblock The octonions.
\newblock {\em Bull. Amer. Math. Soc. (N.S.)}, 39(2):145--205, 2002.

\bibitem[BO81]{benkartosborn:derivations}
G.~M. Benkart and J.~M. Osborn.
\newblock Derivations and automorphisms of nonassociative matrix algebras.
\newblock {\em Trans. Amer. Math. Soc.}, 263(2):411--430, 1981.

\bibitem[CS50]{chevalleyschafer:exceptional}
Claude Chevalley and R.~D. Schafer.
\newblock The exceptional simple {L}ie algebras {$F_4$} and {$E_6$}.
\newblock {\em Proc. Nat. Acad. Sci. U.S.A.}, 36:137--141, 1950.

\bibitem[CS03]{conwaysmith:onquaternions}
John~H. Conway and Derek~A. Smith.
\newblock {\em On quaternions and octonions: their geometry, arithmetic, and
  symmetry}.
\newblock A K Peters, Ltd., Natick, MA, 2003.

\bibitem[Jac58]{jacobson:composition}
N.~Jacobson.
\newblock Composition algebras and their automorphisms.
\newblock {\em Rend. Circ. Mat. Palermo (2)}, 7:55--80, 1958.

\bibitem[Jac60]{jacobson:some}
N.~Jacobson.
\newblock Some groups of transformations defined by {J}ordan algebras. {II}.
  {G}roups of type {$F_{4}$}.
\newblock {\em J. Reine Angew. Math.}, 204:74--98, 1960.

\bibitem[Pet20]{petyt:special}
Harry Petyt.
\newblock The special linear group for nonassociative rings.
\newblock {\em J. Group Theory}, 23(2):327--335, 2020.

\bibitem[Sch95]{schafer:introduction}
Richard~D. Schafer.
\newblock {\em An introduction to nonassociative algebras}.
\newblock Dover Publications, Inc., New York, 1995.
\newblock Corrected reprint of the 1966 original.

\bibitem[Ser95]{serre:cohomologie}
Jean-Pierre Serre.
\newblock Cohomologie galoisienne: progr\`es et probl\`emes.
\newblock In {\em S\'{e}minaire Bourbaki, Vol. 1993/94}, number 227 in
  Ast\'{e}risque, pages 229--257. Soci\'{e}t\'{e} Math\'{e}matique de France,
  1995.

\bibitem[SV00]{springerveldkamp:octonions}
Tonny~A. Springer and Ferdinand~D. Veldkamp.
\newblock {\em Octonions, {J}ordan algebras and exceptional groups}.
\newblock Springer Monographs in Mathematics. Springer-Verlag, Berlin, 2000.

\end{thebibliography}

\end{document}